\documentclass[11pt,reqno]{amsart}

\usepackage{amssymb,latexsym}
\usepackage{url}
\usepackage[all,cmtip]{xy}
\usepackage[hidelinks]{hyperref}
\usepackage{geometry}

\newcommand{\rr}{\mathbb R}
\newcommand{\pp}{\mathbb P}
\newcommand{\bbT}{\mathbb T}
\newcommand{\calC}{\mathcal C}
\newcommand{\calD}{\mathcal D}

\newcommand{\lra}{\longrightarrow}

\newcommand{\subs}{\subset}

\newcommand{\llangle}{\langle\langle}
\newcommand{\rrangle}{\rangle\rangle}
\newcommand{\ver}{\text{ver}}
\newcommand{\pr}{\text{pr}}

\newtheorem{thm}{Theorem}
\newtheorem{prop}{Proposition}
\newtheorem{lemma}{Lemma}

\theoremstyle{definition}
\newtheorem{definition}{Definition}
\newtheorem{remark}{Remark}
\newtheorem{example}{Example}

\begin{document}

\dedicatory{Dedicated to Hern\'an Cendra}

\title[Discrete Mechanical Systems in a Dirac Setting]{Discrete Mechanical Systems in a Dirac Setting: \\ a Proposal}

\author{Mat\'ias I. Caruso}
\address{Depto. de Matem\'atica \\ Facultad de Ciencias Exactas \\ Universidad Nacional de La Plata \hfill \break
Calles 50 y 115 \\ La Plata \\ Buenos Aires \\ 1900 \\ Argentina \hfill \break
\indent Centro de Matem\'atica de La Plata (CMaLP) \hfill \break
\indent CONICET}
\email{mcaruso@mate.unlp.edu.ar}

\author{Javier Fern\'andez}
\address{Instituto Balseiro \\ Universidad Nacional de Cuyo - C.N.E.A. \hfill \break Av. Bustillo 9500 \\ San Carlos de Bariloche \\ R8400AGP \\ Argentina}
\email{jfernand@ib.edu.ar}

\author{Cora Tori}
\address{Depto. de Ciencias B\'asicas \\ Facultad de Ingenier\'ia \\ Universidad Nacional de La Plata \hfill \break	
Calle 116 entre 47 y 48 \\ La Plata \\ Buenos Aires \\ 1900 \\ Argentina \hfill \break
\indent Centro de Matem\'atica de La Plata (CMaLP)}
\email{cora.tori@ing.unlp.edu.ar}

\author{Marcela Zuccalli}
\address{Depto. de Matem\'atica \\ Facultad de Ciencias Exactas \\ Universidad Nacional de La Plata \hfill \break
Calles 50 y 115 \\ La Plata \\ Buenos Aires \\ 1900 \\ Argentina \hfill \break
\indent Centro de Matem\'atica de La Plata (CMaLP)}
\email{marce@mate.unlp.edu.ar}

\thanks{This research was partially supported by grants from Universidad Nacional de Cuyo (\#06/C567 and \#06/C574), Universidad Nacional de La Plata and CONICET}


\begin{abstract}
In these notes, we present an alternative version of discrete Dirac mechanics using Dirac structures. We first establish a notion of `continuous Dirac system' and then propose a definition of discrete Dirac system, proving that it is possible to recover discrete Lagrangian and Hamiltonian systems as particular cases. We also note that this approach allows for kinematic as well as variational constraints.
\end{abstract}

\bibliographystyle{amsalpha}

\maketitle

\section{Introduction}

The notion of Dirac structure as a generalization of both presymplectic and Poisson structures goes back to the work of T. Courant and A. Weinstein around 1990 (\cite{Courant} and \cite{Courant-W}). The usage of these objects to construct mechanical systems can be traced back to, for example, the Port-Hamiltonian systems considered in \cite{D-vdS}. Years later, J. E. Marsden and H. Yoshimura pursued the idea of using Dirac structures to provide a unified framework to treat Lagrangian and Hamiltonian systems (\cite{Y-MI} and \cite{Y-MII}). They introduced the so-called implicit Lagrangian and Hamiltonian systems, which allowed them to consider degenerate Lagrangians and nonholonomic constraints.

In the discrete setting, analogues of these systems were presented by M. Leok and T. Ohsawa in \cite{Leok-O}, where the authors introduced the concept of `discrete induced Dirac structure'. Using these objects, they defined `discrete Lagrange--Dirac systems' and `discrete nonholonomic Hamiltonian systems', that allowed them to recover discrete Lagrangian systems (as considered in \cite{M-West}) and discrete Hamiltonian systems (as considered in \cite{Lall-W}). These structures, however, are not Dirac structures themselves, and it is a natural question whether it is possible to obtain the same results using an actual Dirac structure. The goal of our work is, then, to propose an alternative version of discrete mechanical systems in the Dirac setting, but making use of Dirac structures. We believe that an advantage of this approach would be to, hopefully, make use of results already known in the theory of Dirac structures and attempt, among other things, to construct a reduction procedure.

To this end, instead of treating Lagrangian and Hamiltonian systems separately, we will first establish a notion of `(continuous) Dirac system' (very much as it is done, for example, in \cite{BL-C-GT-MdD} and \cite{C-E-F}) and then present a discrete version of it. This notion of discrete Dirac system will allow us to recover the results of \cite{Leok-O} discussed in the previous paragraph and will provide a little more flexibility, since it will allow for kinematic as well as variational constraints. This is important because many authors, when dealing with discrete mechanical systems, think of the variational and kinematic constraints as being related. In a sense, in the case of systems coming as discretizations of a continuous system, they traditionally arise as two reflections of a single constraint (see, for example, \cite{C-M}). In the discrete setting, however, the kinematic and variational constraints are entirely different objects and we find it unnatural to assume that they have to be related as an intrinsic part of the system. In that sense, we find that the discrete systems correlate easily with the generalized nonholonomic systems of \cite{C-G}.

The outline of this paper is as follows: in Section 2 we review the basic definitions regarding Dirac structures and establish the definition of Dirac system, and in Section 3 we introduce what we call `discrete Dirac systems', showing how they contain both discrete Lagrangian and Hamiltonian systems as particular cases. To conclude, we discuss some topics we would like to study in the future.

\section{Dirac structures and Dirac systems}

\subsection{Dirac structures}

Following the presentation given in \cite{Y-MI} and \cite{Y-MII}, let us first consider a finite-dimensional vector space $V$. If $V^*$ is the dual space, we have the natural duality $\langle \cdot , \cdot \rangle$. We define the symmetric pairing $\llangle \cdot , \cdot \rrangle$ on $V \oplus V^*$ by
$$\llangle (v,\alpha) , (\bar{v},\bar{\alpha}) \rrangle = \langle \alpha , \bar{v} \rangle + \langle \bar{\alpha} , v \rangle,$$
for $(v,\alpha),(\bar{v},\bar{\alpha}) \in V \oplus V^*$.

\begin{definition}
A {\it Dirac structure} on $V$ is a subspace $D \subs V \oplus V^*$ such that $D = D^\perp$, where $D^\perp$ the orthogonal complement of $D$ with respect to $\llangle \cdot , \cdot \rrangle$.
\end{definition}

\begin{definition}
A {\it Dirac structure} on a manifold $M$ is a subbundle $D$ of the {\it Pontryagin bundle} $\bbT M := TM \oplus T^*M$ such that for each $m \in M$, $D(m) \subs T_mM \oplus T_m^*M$ is a Dirac structure in the above sense.
\end{definition}

\begin{remark}
Strictly speaking, the above definition is that of an {\it almost Dirac structure}. The missing feature is an integrability condition that we will omit, as is often the convention when using these objects in mechanics (see, for example, \cite{BL-C-GT-MdD}, \cite{Y-MI}, \cite{Y-MII}).
\end{remark}

The Pontryagin bundle of a manifold $M$ comes with three natural projections: $\rho_{TM} : \bbT M \lra TM$, $\rho_{T^*M} : \bbT M \lra T^*M$ and $\rho_M : \bbT M \lra M$.

A distribution and a two-form on a manifold induce a Dirac structure on it. This result will be used several times in the following pages, so we state it here.

\begin{thm}[Theorem 2.3, \cite{Y-MI}]\label{induced-Dirac}
Let $M$ be a manifold and let $\Omega$ be a two-form on $M$. Given a distribution $\Delta_M$ on $M$, define the skew-symmetric bilinear form $\Omega_{\Delta_M}$ on $\Delta_M$ by restricting $\Omega$ to $\Delta_M$. For each $m \in M$, let
\begin{equation*}
\begin{split}
D_M(m) := \{ (v_m , \alpha_m) \in T_mM \times T^*_mM &\mid v_m \in \Delta_M(m) \text{ and} \\
& \alpha_m(w_m) = \Omega_{\Delta_M}(m)(v_m,w_m) \ \forall w_m \in \Delta_M(m) \}
\end{split}
\end{equation*}
Then, $D_M \subs TM \oplus T^*M$ is a Dirac structure on $M$ which we will denote $D(\Delta_M,\Omega_{\Delta_M})$.
\end{thm}

\subsection{Dirac systems}

Inspired by the definitions considered in, for example, \cite{BL-C-GT-MdD} and \cite{C-E-F}, we introduce the following:

\begin{definition}
A {\itshape Dirac system} consists of a triple $(Q,D,\alpha)$, where $Q$ is a smooth manifold, the configuration space, $D$ is a Dirac structure on $\bbT Q$ and $\alpha$ is a $1$-form on $\bbT Q$, usually arising from the energy of the system.
\end{definition}

\begin{definition}
A curve $z(t)$ on $\bbT Q$ is a {\itshape trajectory} of the system $(Q,D,\alpha)$ if it satisfies
\begin{equation}\label{Dirac-cont}
\dot{z}(t) \oplus \alpha(z(t)) \in D(z(t)).
\end{equation}
\end{definition}

As we see next, these systems naturally contain the implicit Lagrangian and Hamiltonian ones considered in \cite{Y-MI} and \cite{Y-MII}.

\begin{example}[Implicit Lagrangian systems]
An implicit Lagrangian system, as introduced in \cite{Y-MI} (see also \cite{Leok-O}, Definition 2.4), consists of a Lagrangian $L : TQ \lra \rr$, a constraint distribution $\Delta_Q \subs TQ$ and a vector field $X$ on $T^*Q$. Locally, setting\footnote{Throughout this paper we will use a rather compact notation when working in coordinates. Equation $X = \dot{q} \frac{\partial}{\partial q} + \dot{p} \frac{\partial}{\partial p}$ should be interpreted as
$$X = \dot{q}^i \frac{\partial}{\partial q^i} + \dot{p}_i \frac{\partial}{\partial p_i},$$
where we are using the convention of summing over repeated indexes.
} $X = \dot{q} \frac{\partial}{\partial q} + \dot{p} \frac{\partial}{\partial p}$, the trajectories of the system are characterized by the equations (\cite{Y-MI}, Proposition 6.3):
$$p = \frac{\partial L}{\partial v}, \quad \dot{q} \in \Delta_Q(q), \quad v = \dot{q}, \quad \dot{p} - \frac{\partial L}{\partial q} \in \Delta_Q^\circ(q),$$
where $\Delta_Q^\circ(q)$ is the annihilator of $\Delta_Q(q)$.

We can present an implicit Lagrangian system as a Dirac system as follows: given a Lagrangian $L : TQ \lra \rr$ and a constraint distribution $\Delta_Q \subs TQ$, consider the Dirac system $(Q,D_\Delta,\alpha_L)$ given by
\begin{itemize}
\item[$\ast$] $D_\Delta := D(\Delta,\omega_P)$, a Dirac structure on $\bbT Q$ induced by $\Delta := (T\rho_Q)^{-1}(\Delta_Q)$ and $\omega_P := \rho_{T^*Q}^* \omega_Q$, where $\omega_Q$ is the canonical symplectic structure on $T^*Q$,
\item[$\ast$] $\alpha_L := d\mathcal{E}_L$, with $\mathcal{E}_L : \bbT Q \lra \rr$ defined as $\mathcal{E}_L(v_q,\alpha_q) := \alpha_q(v_q) - L(v_q).$
\end{itemize}

Explicitly, the Dirac structure $D_\Delta$ is described as
$$D_\Delta(z) = \{ w_z \oplus \beta_z \in T_z \bbT Q \oplus T_z^* \bbT Q \mid w_z \in \Delta(z), \ \beta_z - i_{w_z}\omega_P \in \Delta^\circ(z) \}.$$

In canonical coordinates,
$$\alpha_L = - \frac{\partial L}{\partial q} \ dq + \left( p - \frac{\partial L}{\partial v} \right) \ dv + v \ dp$$
and condition (\ref{Dirac-cont}) is equivalent to $(q,\dot{q}) \in \Delta_Q$ and
$$- \frac{\partial L}{\partial q} \ \delta q + \left( p - \frac{\partial L}{\partial v} \right) \ \delta v + v \ \delta p = \dot{q} \ \delta p - \dot{p} \ \delta q,$$
for all $(\delta q, \delta v, \delta p) \in \Delta(q,v,p)$, where
\begin{equation}\label{Delta(q,v,p)}
\Delta(q,v,p) = \{ (\delta q, \delta v, \delta p) \in T_{(q,v,p)} \bbT Q \mid \delta q \in \Delta_Q(q) \}.
\end{equation}

Therefore, the equations of $(Q,D_\Delta,\alpha_L)$ are
$$\dot{p} - \frac{\partial L}{\partial q} \in \Delta_Q^\circ, \quad p = \frac{\partial L}{\partial v}, \quad \dot{q} = v, \quad (q,\dot{q}) \in \Delta_Q,$$
which are precisely those of the original implicit Lagrangian system.
\end{example}

\begin{example}[Implicit Hamiltonian systems]
An implicit Hamiltonian system, as considered in \cite{Y-MII} (see also \cite{Leok-O}, Definition 2.7), consists of a Hamiltonian $H : T^*Q \lra \rr$, a constraint distribution $\Delta_Q \subs TQ$ and a vector field $X$ on $T^*Q$. In canonical coordinates, setting $X = \dot{q} \frac{\partial}{\partial q} + \dot{p} \frac{\partial}{\partial p}$, the trajectories of the system are characterized by the equations (\cite{Y-MII}, Proposition 3.14):
$$\dot{p} + \frac{\partial H}{\partial q} \in \Delta_Q^\circ, \quad \dot{q} = \frac{\partial H}{\partial p}, \quad (q,\dot{q}) \in \Delta_Q.$$

In the Dirac setting, given a Hamiltonian $H : T^*Q \lra \rr$ and a constraint distribution $\Delta_Q \subs TQ$, consider the Dirac system $(Q,D_\Delta,d\tilde{H})$ where $\tilde{H} := H \circ \pr_{T^*Q}$ and $D_\Delta$ is defined as in Example 1.

This time, condition (\ref{Dirac-cont}) reads
$$\left( \dot{p} + \frac{\partial H}{\partial q} \right) \ \delta q + \left( - \dot{q} + \frac{\partial H}{\partial p} \right) \ \delta p = 0$$
for all $(\delta q,\delta v,\delta p) \in \Delta(q,v,p)$ (where, as in Example 1, $\Delta(q,v,p)$ is given by (\ref{Delta(q,v,p)})), together with $(q,\delta q) \in \Delta_Q$. Therefore, the equations of the system are
$$\dot{p} + \frac{\partial H}{\partial q} \in \Delta_Q^\circ, \quad \dot{q} = \frac{\partial H}{\partial p}, \quad (q,\dot{q}) \in \Delta_Q,$$
which are the equations of the original implicit Hamiltonian system.
\end{example}

\section{Discrete Dirac systems}

We now turn our attention to the discrete setting, presenting a notion of `discrete Dirac system'. A discrete version of the implicit Lagrangian and Hamiltonian systems considered in \cite{Y-MI} and \cite{Y-MII} by H. Yoshimura and J. E. Marsden was presented in \cite{Leok-O} by M. Leok and T. Ohsawa. In their work, they introduce an object they call `discrete induced Dirac structure', which is a discrete analogue of the structures described in Theorem \ref{induced-Dirac}, but fails to be a Dirac structure in itself (it is not a subbundle of the corresponding Pontryagin bundle).

Our goal is then to present an alternative approach that makes use of actual Dirac structures. In order to do so, we will not consider implicit Lagrangian and Hamiltonian systems separately, but will construct a discrete analogue of (continuous) Dirac systems that contains both discrete Lagrangian and Hamiltonian systems as particular cases.

In discrete mechanics, one often works with maps whose domains are product manifolds, giving rise to two natural operators, namely $D_1$ and $D_2$. We briefly review this construction below, before using it in the following sections.

Given an $n$-dimensional manifold $Q$, we consider the product manifold $Q \times Q$ and, for $i = 1,2$, the projection $\pr_i : Q \times Q \lra Q$ onto the $i$-th factor. Using the product structure of $Q \times Q$, we have that
$$T(Q \times Q) \simeq  \pr_1^* (TQ) \oplus \pr_2^* (TQ),$$
where $\pr_i^* (TQ)$ denotes the pullback of the tangent bundle $TQ \lra Q$ over $Q \times Q$ by $\pr_i$ for $i=1,2$. If we define $j_1 : \pr_1^* (TQ) \lra T(Q \times Q)$ as $j_1(\delta q) := (\delta q,0)$, we have that $j_1$ is an isomorphism of vector bundles between $\pr_1^*  (TQ)$ and $TQ^{-} := \ker(T \pr_2) \subset T(Q \times Q)$. Similarly,
defining $j_2 : \pr_2^*  (TQ) \lra T(Q \times Q)$ as $j_2(\delta q) := (0,\delta q)$ identifies $\pr_2^* (TQ)$ with the subbundle
$TQ^{+} := \ker(T \pr_1) \subset T(Q \times Q)$. 

So, the decomposition $T(Q \times  Q) = TQ^{-} \oplus TQ^{+}$ leads to the decomposition
$$T^*(Q \times Q) = (TQ^{-})^{\circ} \oplus (TQ^{+})^{\circ}$$
and the natural identifications
$(TQ^{+})^{\circ} \simeq (TQ^{-})^{*} \simeq \pr_1^* T^{*}Q$ and 
$(TQ^{-})^{\circ} \simeq (TQ^{+})^{*} \simeq \pr_2^* T^{*}Q$.

For any smooth map
$f : Q \times Q \rightarrow X$, where $X$ is a smooth manifold, we define $D_1 f := Tf \circ j_1$ and $D_2 f := Tf \circ j_2$, where $Tf : T(Q \times Q) \lra TX$ denotes the tangent map of $f$. Thus,
$$Tf(q_0,q_1) (\delta q_0,\delta q_1) = D_1f(q_0,q_1)(\delta q_0) + D_2f(q_0,q_1)(\delta q_1).$$

In particular, if $f : Q \times Q \lra \rr$, then $D_1f(q_0,q_1) \in T^*_{q_0}
Q$ and $D_2f(q_0,q_1) \in T^*_{q_1}Q$.

Let $M$ be a smooth manifold. Given a natural number $N$, a {\it discrete curve} of length $N$ is a map $x. : \{ 0,\ldots,N \} \lra M$. Note that the space of discrete curves of length $N$ may be identified with the cartesian product $M^{N+1}$, so it has a (finite dimensional) smooth manifold structure.

\begin{definition}
Let $Q$ be a smooth manifold. We define its {\it discrete Pontryagin bundle} as
\begin{equation}\label{Pd}
\pp^d_Q := (Q \times Q) \times_Q T^*Q \simeq T^*Q \times Q,
\end{equation}
where we are considering the fiber bundles $\pr_1 : Q \times Q \lra Q$ and the cotangent bundle $\pi_Q : T^*Q \lra Q$.
\end{definition}

Intuitively, the first space in (\ref{Pd}) is more natural, since $Q \times Q$ is the discrete analogue of the tangent bundle\footnote{It is a well established idea to replace tangent vectors with close enough points in $Q$ when considering discrete-time dynamical systems.}, but we will use the latter because it is easier to work on a product manifold. The two natural projections are $\pr_{T^*Q} : \pp^d_Q \lra T^*Q$ and $\pr_Q : \pp^d_Q \lra Q$, given by $\pr_Q(\alpha_q,q^+) := q^+$.

We will always use the following notation unless explicitly stated otherwise: $x = (\alpha_q,q^+)$ will be a point in $\pp^d_Q$, with $\alpha_q \in T_q^*Q$ and $q^+ \in Q$, and $\delta x = (\delta \alpha_q,\delta q^+)$ will be a tangent vector to $\pp^d_Q$ at $x$, with $\delta \alpha_q \in T_{\alpha_q}T^*Q$ and $\delta q^+ \in T_{q^+}Q$. A discrete curve on $\pp^d_Q$ will therefore be denoted $x. = (\alpha_{q.},q^+.)$. Notice that the subindexes are used to indicate the position in a given path and have nothing to do with coordinates.

We will restrict our attention to the discrete curves on $\pp^d_Q$ that satisfy a certain `second order condition'. We say that a discrete curve $x.$ is {\itshape admissible} if
$$(x_k,x_{k+1}) = ((\alpha_{q_k},q_k^+),(\alpha_{q_{k+1}},q_{k+1}^+))$$ satisfies $q_k^+ = q_{k+1}$ for all $k = 0,\ldots,N-1$. This means that $x_k = (\alpha_{q_k},q_{k+1})$ for all $k$.

We denote the space of admissible discrete curves on $\pp^d_Q$ of length $N$ by $\calC_d(\pp^d_Q)$.

For our definition of discrete Dirac system we will use $1$-forms defined on $\calC_d(\pp^d_Q)$, so we first check that this makes sense.

\begin{lemma}
$\calC_d(\pp^d_Q)$ is a regular submanifold of $\pp^d_Q \times \ldots \times \pp^d_Q$.
\end{lemma}
\begin{proof}
Let us define a map $\phi : (\pp^d_Q)^{N+1} \lra (Q \times Q)^{N+1}$, where $(\pp^d_Q)^{N+1}$ means $\pp^d_Q \times \ldots \times \pp^d_Q$ N+1 times and similarly for $(Q \times Q)^{N+1}$, by
$$\phi((\alpha_{q_0},q_0^+),\ldots,(\alpha_{q_N},q_N^+)) := ((q_0,q_0^+),\ldots,(q_N,q_N^+)).$$

Notice that $\phi$ is a submersion and $\calC_d(\pp^d_Q)$ can be regarded as
$$\calC_d(\pp^d_Q) = \phi^{-1}(Q \times \underbrace{\Delta(Q) \times \ldots \times \Delta(Q)}_{N \text{ times}} \times Q),$$
where $\Delta(Q) \subs Q \times Q$ is the diagonal, which is a regular submanifold. Therefore, $Q \times \Delta(Q) \times \ldots \times \Delta(Q) \times Q$ is a regular submanifold and $\calC_d(\pp^d_Q)$ is its nonempty preimage by a submersion.
\end{proof}

For the rest of this article, we will use the following notation: since
$$T_{x.}(\pp^d_Q \times \ldots \times \pp^d_Q) \simeq T_{x_0}\pp^d_Q \oplus \ldots \oplus T_{x_N}\pp^d_Q,$$
a $1$-form $\psi$ on $\pp^d_Q \times \ldots \times \pp^d_Q$ can be regarded as a sum $\psi = \psi_0 + \ldots + \psi_N$, with $\psi_j : (\pp^d_Q)^{N+1} \lra T^*\pp^d_Q$ such that $\psi_j(x_0,\ldots,x_N) \in T^*_{x_j} \pp^d_Q$. Therefore, we will denote $\psi_k(x.) := \psi(x.,k)$.

\begin{definition}
A {\it discrete Dirac system} is given by $(Q,D,\calD,\psi)$, where $Q$ is a smooth manifold, $D$ is a Dirac structure on $\pp^d_Q$, $\calD \subs Q \times Q$ is a submanifold and $\psi$ is a $1$-form on $\calC_d(\pp^d_Q)$.
\end{definition}

\begin{definition}
A discrete curve $x. = (\alpha_{q.},q^+.) \in \calC_d(\pp^d_Q)$ on $\pp^d_Q$ is a {\it trajectory} of the system if it satisfies $(q_k,q_{k+1}) \in \calD$ together with
\begin{equation}\label{Dirac-disc}
\left( \ver_{\alpha_{q_k}}(\alpha_{q_k}) , 0 \right) \oplus \psi_k(x.) \in D(x_k), \quad 0 \le k \le N-1
\end{equation}
where $\ver$ is the vertical lift on $T^*Q$ given by
$$\ver_{\alpha_q}(\beta_q) := \frac{d}{dt} \bigg|_{t=0} \alpha_q + t \beta_q \in T_{\alpha_q}T^*Q.$$
\end{definition}

In local coordinates, equation (\ref{Dirac-disc}) is simply
$$(0,p_k,0) \oplus \psi_k(x.) \in D(x_k), \quad \quad 0 \le k \le N-1.$$

In the following pages we will see how this formulation contains both discrete Lagrangian and Hamiltonian systems. To this end, we will consider the $2$-form $\omega_P^d := -\pr_{T^*Q}^* \omega_Q$ on $\pp^d_Q$, where $\omega_Q$ is the canonical symplectic form on $T^*Q$.

In canonical coordinates, given $\delta x_k := (\delta \alpha_{q_k} , \delta q_k^+) \in T_{x_k} \pp^d_Q$,
$$\omega_P^d(x_k) \left( \left( \ver_{\alpha_{q_k}}(\alpha_{q_k}) , 0 \right) , \delta x_k \right) = -\omega_Q(q_k,p_k)((0,p_k) , (\delta q_k,\delta p_k)) = p_k \cdot \delta q_k.$$

\subsection{Discrete Lagrangian systems}\label{section-DLS}

A well known type of discrete-time constrained mechanical system is the family of nonholonomic discrete mechanical systems (see \cite{C-M} and \cite{F-T-Z}). One such system consists of $(Q,L_d,\Delta_Q,\calD)$ where $Q$ is a smooth manifold, $L_d : Q \times Q \lra \rr$ is a smooth function, $\Delta_Q$ is a subbundle of $TQ$ and $\calD \subs Q \times Q$ is a submanifold. The trajectories of these systems are characterized by the following equations:
\begin{equation}\label{NHDMS}
D_1 L_d(q_k,q_{k+1}) + D_2 L_d(q_{k-1},q_k) \in \Delta_Q^\circ(q_k), \quad (q_k,q_{k+1}) \in \calD.
\end{equation}

We can present a nonholonomic discrete mechanical system as a discrete Dirac system as follows: given a discrete Lagrangian $L_d : Q \times Q \lra \rr$, a constraint distribution $\Delta_Q$ and a submanifold $\calD \subs Q \times Q$, consider the discrete Dirac system $(Q,D_\Delta,\calD,\psi_L)$, where $D_\Delta := D(\Delta,\omega_P^d)$ is the Dirac structure induced by $\omega_P^d$ and $\Delta := (T(\pi_Q \circ \pr_{T^*Q}))^{-1}(\Delta_Q)$, and $\psi_L$ is defined as
$$\psi_L(x.,k) := - D_1 L_d(q_k,q_{k+1}) + (\alpha_{q_{k+1}} - D_2 L_d(q_k,q_{k+1})).$$

The distribution $\Delta$ is given explicitly by
$$\Delta(x) = \{ (\delta \alpha_q,\delta q^+) \in T_x\pp^d_Q \mid \delta q := T\pi_Q (\delta \alpha_q) \in \Delta_Q(q) \}.$$

In canonical coordinates, condition (\ref{Dirac-disc}) is equivalent to $(0,p_k,0) \in \Delta(x_k)$ together with
$$p_k \cdot \delta q_k = -D_1 L_d(q_k,q_k^+) \cdot \delta q_k + \left( p_{k+1} - D_2 L_d(q_k,q_k^+) \right) \cdot \delta q_k^+,$$
for all $\delta x_k = (\delta q_k,\delta p_k,\delta q_k^+) \in \Delta(x_k)$. That is,
$$(p_k + D_1L_d(q_k,q_k^+)) \cdot \delta q_k - (p_{k+1} - D_2L_d(q_k,q_k^+)) \cdot \delta q_k^+ = 0,$$
for all $\delta x_k \in \Delta(x_k)$, and $(0,p_k,0) \in \Delta(x_k)$, which is equivalent to
$$T\pi_Q \cdot (q_k,p_k,0,p_k) = (q_k,0) \in \Delta_Q,$$
and is trivially satisfied, because $\Delta_Q(q_k)$ is vector subspace of $T_{q_k}Q$.

Therefore, the equations of the system are
\begin{eqnarray*}
	p_{k+1} = D_2L_d(q_k,q_k^+), &\quad& p_k + D_1L_d(q_k,q_k^+) \in \Delta_Q^\circ (q_k) \\
	q_k^+ = q_{k+1}, &\quad& (q_k,q_{k+1}) \in \calD.
\end{eqnarray*}

That is,
\begin{equation}\label{Lagrange-d}
p_{k+1} = D_2L_d(q_k,q_{k+1}), \quad p_k + D_1L_d(q_k,q_{k+1}) \in \Delta_Q^\circ (q_k), \quad (q_k,q_{k+1}) \in \calD.
\end{equation}

Comparison of (\ref{Lagrange-d}) with (\ref{NHDMS}) proves the following:

\begin{prop}
The trajectories of the nonholonomic discrete mechanical system\\ $(Q,L_d,\Delta_Q,\calD)$ correspond to those of the discrete Dirac system $(Q,D_\Delta,\calD,\psi_L)$.
\end{prop}

\begin{remark}
If the submanifold $\calD$ and the distribution $\Delta_Q$ are related by means of a retraction on $Q$ (see \cite{Leok-O}, Section 4), then equations (\ref{Lagrange-d}) are equivalent to the $(+)$-discrete Lagrange--Dirac equations considered in \cite{Leok-O}, Section 5.1.
\end{remark}

\begin{remark}\label{initial-data}
Notice that without constraints, equation $p_0 + D_1L_d(q_0,q_1) = 0$ forces an initial condition $x_0 \in \pp^d_Q$ to be of the form $(-D_1L_d(q_0,q_1),q_1)$ if we expect it to give rise to a trajectory.
\end{remark}

\begin{example}
Let us consider the discrete harmonic oscillator described by the (unconstrained) discrete Lagrangian system $(\rr,L_d)$, where
$$L_d(q,q^+) := h \left[ \frac{1}{2} \left( \frac{q^+ - q}{h} \right)^2 - \frac{\lambda}{2} q^2 \right],$$
and $\lambda$ is a nonnegative constant.

The derivatives we will need are
$$D_1L_d(q,q^+) = - \frac{1}{h} (q^+ - q) - h \lambda q, \quad D_2L_d(q,q^+) = \frac{1}{h} (q^+ - q).$$

In the Dirac setting, we consider the system $(\rr,D,\rr \times \rr,\psi_L)$, where
$$\psi_L(x.,k) = \left( \frac{q_{k+1} - q_k}{h} + h \lambda q_k , 0 , p_{k+1} - \frac{q_{k+1} - q_k}{h} \right) \in T_{q_k}^*\rr \times T_{p_k}^*\rr \times T_{q_{k+1}}^*\rr$$
and
\begin{equation*}
\begin{split}
D(x) &= \{ v \oplus \alpha \mid \omega_P^d(v,w) = \alpha(w) \ \forall w \in T_x\pp^d_Q \} \\
&= \{ (\dot{q},\dot{p},q') \oplus \alpha \mid \dot{p} \delta q - \dot{q} \delta p = \alpha(\delta q, \delta p, \delta q^+) \ \forall (\delta q, \delta p, \delta q^+) \in T_x\pp^d_Q \}.
\end{split}
\end{equation*}

The initial data $(q_0,q_1) \in \rr \times \rr$ for the original Lagrangian system induces the initial data $x_0 = (q_0,p_0,q_1)$, with $(q_0,p_0) = -D_1L_d(q_0,q_1)$ (see Remark \ref{initial-data}). Let us construct a trajectory $(x_0,x_1)$.

The condition (\ref{Dirac-disc}) for $k=0$ is $(0,p_0,0) \oplus \psi_L(x.,0) \in D(x_0)$, that is,
$$p_0 \ \delta q_0 = \left( \frac{q_1 - q_0}{h} + h \lambda q_0 \right) \ \delta q_0 + \left( p_1 - \frac{q_1 - q_0}{h} \right) \ \delta q_1,$$
for all $(\delta q_0,\delta p_0, \delta q_1) \in T_{x_0}\pp^d_\rr$. This yields
$$p_0 = \frac{q_1 - q_0}{h} + h \lambda q_0, \quad p_1 = \frac{q_1 - q_0}{h}.$$

The first equation is the restriction regarding the initial data that we discussed in Remark \ref{initial-data}. Iterating, when $k=1$ we have $(0,p_1,0) \oplus \psi_L(x.,1) \in D(x_1)$:
$$\frac{q_1 - q_0}{h} \ \delta q_1 = \left( \frac{q_2 - q_1}{h} + h \lambda q_1 \right) \ \delta q_1 + \left( p_2 - \frac{q_2 - q_1}{h} \right) \ \delta q_2,$$
for all $(\delta q_1,\delta p_1, \delta q_2) \in T_{x_1}\pp^d_\rr$. This leads to
$$q_2 = 2 q_1 - q_0 + h^2 \lambda q_1, \quad p_2 = \frac{q_2 - q_1}{h}.$$

Therefore, the trajectory $(x_0,x_1)$ is
$$(x_0,x_1) = \left( \left( q_0 , \frac{q_1 - q_0}{h} + h \lambda q_0 , q_1 \right) , \left( q_1 , \frac{q_1 - q_0}{h} , 2 q_1 - q_0 - h^2 \lambda q_1 \right) \right).$$
\end{example}

\subsection{Discrete Hamiltonian systems}

In \cite{Leok-Z}, the authors introduce the notion of discrete Hamiltonian mechanical system from a variational viewpoint. Such a system is given by a vector space $Q$ and a discrete Hamiltonian function $H_d : T^*Q \simeq Q \times Q^* \lra \rr$. Its trajectories are characterized by the `discrete right Hamilton's equations'\footnote{We are omitting the adjective `right' everywhere else, since we are using $H_d$ instead of the original $H_d^+$ considered in \cite{Leok-Z}.}:
\begin{equation}\label{dHMS}
q_{k+1} = D_2H_d(q_k,p_{k+1}), \quad p_k = D_1H_d(q_k,p_{k+1}).
\end{equation}

In the Dirac setting, given a vector space $Q$, a discrete Hamiltonian $H_d : Q \times Q^* \lra \rr$, a constraint distribution $\Delta_Q \subs TQ$ and a submanifold $\calD \subs Q \times Q$, consider the system $(Q,D_\Delta,\calD,\psi_H)$, where $D_\Delta := D(\Delta,\omega_P^d)$ is, as in the Lagrangian systems, the Dirac structure induced by $\omega_P^d$ and $\Delta := (T(\pi_Q \circ \pr_{T^*Q}))^{-1}(\Delta_Q)$, and the $1$-form $\psi_H$ is given in global canonical coordinates by
$$\psi_H \left( x. , k \right) := \frac{\partial H_d}{\partial q}(q_k,p_{k+1}) \ dq_k + \left( \frac{\partial H_d}{\partial p}(q_k,p_{k+1}) - q_{k+1} \right) \ dp_k.$$

In these coordinates, equation (\ref{Dirac-disc}) is equivalent to $(0,p_k,0) \in \Delta(x_k)$ (which is trivially satisfied, as in Section \ref{section-DLS}) and
$$p_k \cdot \delta q_k = \frac{\partial H_d}{\partial q}(q_k,p_{k+1}) \cdot \delta q_k + \left( \frac{\partial H_d}{\partial p}(q_k,p_{k+1}) - q_{k+1} \right) \cdot \delta p_k,$$
for all $\delta x_k = (\delta q_k,\delta p_k,\delta q_{k+1}) \in \Delta(x)$. That is,
$$\left( p_k - \frac{\partial H_d}{\partial q}(q_k,p_{k+1}) \right) \cdot \delta q_k - \left( \frac{\partial H_d}{\partial p}(q_k,p_{k+1}) - q_{k+1} \right) \cdot \delta p_k = 0$$
for all $\delta x_k \in \Delta(x_k)$.

Therefore, the equations of the system are
\begin{equation}\label{Hamilton-d}
p_k - \frac{\partial H_d}{\partial q}(q_k,p_{k+1}) \in \Delta_Q^\circ(q_k), \quad q_{k+1} = \frac{\partial H_d}{\partial p}(q_k,p_{k+1}), \quad (q_k,q_{k+1}) \in \calD.
\end{equation}

If $\Delta_Q = TQ$ and $\calD = Q \times Q$, comparing equations (\ref{Hamilton-d}) and (\ref{dHMS}) proves the following result:

\begin{prop}
The trajectories of the (unconstrained) discrete Hamiltonian system $(Q,H_d)$ correspond to those of the discrete Dirac system $\left( Q,D_{T \pp^d_Q},Q \times Q, \psi_H \right)$.
\end{prop}

\begin{remark}
As before, if $\calD$ and $\Delta_Q$ are linked via a retraction on $Q$, equations (\ref{Hamilton-d}) are equivalent to the $(+)$-discrete nonholonomic Hamilton's equations considered in \cite{Leok-O}.
\end{remark}

\begin{remark}
If $\Delta_Q = TQ$ and $\calD = Q \times Q$, under certain regularity conditions, namely
$$\frac{\partial^2 H_d}{\partial q \partial p}(q,p)$$
being nonsingular (see \cite{Leok-Z}), in a neighbourhood of a solution $(\bar{q}_0,\bar{p}_0,\bar{p}_1)$, equation $p_0 = \frac{\partial H_d}{\partial q}(q_0,p_1)$ determines implicitly the value $p_1$ in terms of $q_0$ and $p_0$ and, afterwards, $q_1 = \frac{\partial H_d}{\partial p}(q_0,p_1)$ determines $q_1$. Schematically, $(q_0,p_0) \dashrightarrow p_1 \dashrightarrow q_1$, meaning that, under certain conditions, $q_1$ is a function of $(q_0,p_0)$ and, hence, cannot be arbitrary. This says, again, that not every choice of $x_0 = (q_0,p_0,q_1)$ gives rise to a trajectory.
\end{remark}

\section*{Future work}

\begin{itemize}
\item[$\ast$] The dynamics of implicit Lagrangian and Hamiltonian systems can be obtained from a variational principle, as it is shown in \cite{Y-MII}. For the case of discrete Lagrange--Dirac and nonholonomic Hamiltonian systems considered in \cite{Leok-O}, the authors show that their dynamics can also be derived using variational techniques. We are interested in studying if it is possible to do the same for the discrete Dirac systems that we considered here.

\item[$\ast$] Under some regularity condition, the existence of flows near a given trajectory is a well known fact for both discrete Lagrangian and Hamiltonian systems. We would like to explore the possibility of a similar result in the context of the discrete Dirac systems introduced here.

\item[$\ast$] In the continuous setting, reduction of Dirac structures and implicit Lagrangian and Hamiltonian systems is discussed, for example, in \cite{Y-M-Red} and \cite{Y-M-RedII}. We are interested in studying symmetries of discrete Dirac systems, and in constructing a reduction procedure in this new framework.
\end{itemize}

\providecommand{\bysame}{\leavevmode\hbox to3em{\hrulefill}\thinspace}
\providecommand{\MR}{\relax\ifhmode\unskip\space\fi MR }
\providecommand{\MRhref}[2]{%
  \href{http://www.ams.org/mathscinet-getitem?mr=#1}{#2}
}
\providecommand{\href}[2]{#2}


\end{document}